\documentclass[12pt,reqno]{amsart}
\usepackage{amsthm}
\usepackage{amssymb}
\usepackage{graphics}
\usepackage{tikz}
\usetikzlibrary{shapes,backgrounds,calc}
\usepackage{latexsym}
\usepackage{multicol}
\usepackage{verbatim,enumerate}
\usepackage{accents}
\usepackage{cite}
\usepackage{multirow}
\usepackage{bigstrut}
\usepackage{amsthm}
\usepackage{amssymb}
\usepackage{graphics}
\usepackage{tikz}
\usetikzlibrary{shapes,backgrounds,calc}
\usepackage{latexsym}
\usepackage{multicol}
\usepackage{verbatim,enumerate}
\usepackage{accents}
\usepackage{cite}
\usepackage{array}
\usepackage[colorlinks=true, linkcolor=blue, citecolor=blue, urlcolor=black]{hyperref}
\usepackage{hyperref}
\usepackage{amsmath, amscd,url}
\usepackage{multirow}
\usepackage{bigstrut}
\usepackage{longtable}

\usepackage{stackengine}

\usepackage{hyperref}
\usepackage{amsmath, amscd,url}
\usepackage{longtable}

\advance\textwidth by 1.3in \advance\oddsidemargin by -.6in \advance\evensidemargin by -.6in
\theoremstyle{definition}

\newtheorem{theorem}{Theorem}[section]
\newtheorem{lemma}[theorem]{Lemma}
\newtheorem{proposition}[theorem]{Proposition}
\newtheorem{corollary}[theorem]{Corollary}
\theoremstyle{definition}

\newtheorem{example}[theorem]{Example}

\theoremstyle{remark}

\theoremstyle{definition}

\newcounter{cnt}
 \makeatletter
\def\mydggeometry{\makeatletter\dg@YGRID=1\dg@XGRID=20\unitlength=0.003pt\makeatother}
\makeatother \theoremstyle{remark}

\numberwithin{equation}{section}
\let\bwdg\bigwedge
\def\bigwedge{{\textstyle\bwdg}}

\newcommand{\nc}{\newcommand}
\newcommand{\rnc}{\renewcommand}

\nc{\cal}{\mathcal} \nc{\goth}{\mathfrak} \rnc{\bold}{\mathbf}

\nc\bomega{{\mbox{\boldmath $\omega$}}} \nc\bpsi{{\mbox{\boldmath $\Psi$}}}
 \nc\balpha{{\mbox{\boldmath $\alpha$}}}
 \nc\bpi{{\mbox{\boldmath $\pi$}}}
 \nc\bvpi{{\mbox{\boldmath $\varpi$}}}
\nc\chara{\operatorname{ch}}

  \nc\bxi{{\mbox{\boldmath $\xi$}}}
\nc\bmu{{\mbox{\boldmath $\mu$}}} \nc\bcN{{\mbox{\boldmath $\cal{N}$}}} \nc\bcm{{\mbox{\boldmath $\cal{M}$}}} \nc\blambda{{\mbox{\boldmath
$\lambda$}}}\nc\bnu{{\mbox{\boldmath $\nu$}}}

\makeatletter
\def\section{\def\@secnumfont{\mdseries}\@startsection{section}{1}%
  \z@{.7\linespacing\@plus\linespacing}{.5\linespacing}%
  {\normalfont\scshape\centering}}
\def\subsection{\def\@secnumfont{\bfseries}\@startsection{subsection}{2}%
  {\parindent}{.5\linespacing\@plus.7\linespacing}{-.5em}%
  {\normalfont\bfseries}}
\makeatother

 \nc{\Hom}{\operatorname{Hom}}
  \nc{\mode}{\operatorname{mod}}
\nc{\End}{\operatorname{End}} \nc{\wh}[1]{\widehat{#1}} \nc{\Ext}{\operatorname{Ext}} \nc{\ch}{\text{ch}} \nc{\ev}{\operatorname{ev}}
\nc{\Ob}{\operatorname{Ob}} \nc{\soc}{\operatorname{soc}} \nc{\rad}{\operatorname{rad}} \nc{\head}{\operatorname{head}}

 \nc{\Cal}{\cal} \nc{\Xp}[1]{X^+(#1)} \nc{\Xm}[1]{X^-(#1)}
\nc{\on}{\operatorname} \nc{\Z}{{\bold Z}} \nc{\J}{{\cal J}}  \nc{\Q}{{\bold Q}}

\nc{\N}{{\bold N}}  \nc\boa{\bold a} \nc\bob{\bold b} \nc\boc{\bold c} \nc\bod{\bold d} \nc\boe{\bold e} \nc\bof{\bold f} \nc\bog{\bold g}
\nc\boh{\bold h} \nc\boi{\bold i} \nc\boj{\bold j} \nc\bok{\bold k} \nc\bol{\bold l} \nc\bom{\bold m} \nc\bon{\mathbb n} \nc\boo{\bold o}
\nc\bop{\bold p} \nc\boq{\bold q} \nc\bor{\bold r} \nc\bos{\bold s} \nc\boT{\bold t} \nc\boF{\bold F} \nc\bou{\bold u} \nc\bov{\bold v}
\nc\bow{\bold w} \nc\boz{\bold z}\nc\ba{\bold A} \nc\bb{\bold B} \nc\bc{\mathbb C} \nc\bd{\bold D} \nc\be{\bold E} \nc\bg{\bold
G} \nc\bh{\bold H} \nc\bi{\bold I} \nc\bj{\bold J} \nc\bk{\bold K} \nc\bl{\bold L} \nc\bm{\bold M} \nc\bn{\mathbb N} \nc\bo{\bold O} \nc\bp{\bold
P} \nc\bq{\bold Q} \nc\br{\bold R} \nc\bs{\bold S} \nc\bt{\bold T} \nc\bu{\bold U} \nc\bv{\bold V} \nc\bw{\bold W} \nc\bz{\mathbb Z} \nc\bx{\bold
x} \nc\KR{\bold{KR}} \nc\rk{\bold{rk}} \nc\het{\text{ht }}

\nc\toa{\tilde a} \nc\tob{\tilde b} \nc\toc{\tilde c} \nc\tod{\tilde d} \nc\toe{\tilde e} \nc\tof{\tilde f} \nc\tog{\tilde g} \nc\toh{\tilde h}
\nc\toi{\tilde i} \nc\toj{\tilde j} \nc\tok{\tilde k} \nc\tol{\tilde l} \nc\tom{\tilde m} \nc\ton{\tilde n} \nc\too{\tilde o} \nc\toq{\tilde q}
\nc\tor{\tilde r} \nc\tos{\tilde s} \nc\toT{\tilde t} \nc\tou{\tilde u} \nc\tov{\tilde v} \nc\tow{\tilde w} \nc\toz{\tilde z} \nc\woi{w_{\omega_i}}

\begin{document}
\setcounter{section}{0}
\setcounter{tocdepth}{1}


\title{On power basis of a class of number fields}

\author[Anuj Jakhar]{Anuj Jakhar}
\author[Sumandeep Kaur]{Sumandeep Kaur}
\author[Surender Kumar]{Surender Kumar}
\address[Anuj Jakhar]{Department of Mathematics, Indian Institute of Technology (IIT) Madras}
\address[Sumandeep Kaur]{Department of Mathematics, Panjab University Chandigarh}
\address[Surender Kumar]{Department of Mathematics, Indian Institute of Technology (IIT) Bhilai}

\email[Anuj Jakhar]{anujjakhar@iitm.ac.in \\ anujiisermohali@gmail.com}
\email[Sumandeep Kaur]{sumandhunay@gmail.com}
\email[Surender Kumar]{surenderk@iitbhilai.ac.in}



\begin{abstract}
\noindent  Let $f(x)=x^n+ax^2+bx+c \in \Z[x]$ be an irreducible polynomial with $b^2=4ac$ and let $K=\Q(\theta)$ be an algebraic number field defined by a complex root $\theta$ of $f(x)$. Let $\Z_K$ deonote the ring of algebraic integers of $K$. The aim of this paper is to provide the necessary and sufficient conditions involving only $a,c$ and $n$ for a given prime $p$ to divide the index of the subgroup $\Z[\theta]$ in $\Z_K$. As a consequence, we provide  families of monogenic algebraic number fields. Further, when $\Z_K \neq \Z[\theta]$, we determine explicitly the index $[\Z_K : \Z[\theta]]$ in some cases.
\end{abstract}
\maketitle

\section{Introduction and statements of results}\label{intro}
An algebraic number field $K$ of degree $n$ is said to be monogenic if it possess an integral basis of the form $\{1, \alpha, \alpha^2,\cdots,\alpha^{n-1}\}$ for some algebraic integer $\alpha.$ Such an integral basis of $K$ is called a power basis of $K$. The determination of monogenity of an algebraic number field is one of the classical problem in algebraic number theory. Hasse put forth the problem of characterising monogenic fields, and it is still an active field of research. Despite the lack of a standard classification, there has been some progress in this direction. Ga\'al and Gr\'as have done extensive research on the idea of monogenic fields. Results relevant to lower degree monogenic extensions can be found in Ga\'al's book \cite{gaal}. Gr\'as demonstrates in \cite{gras}  that for $n$ relatively prime to 6, there are only finitely many abelian monogenic
fields of degree $n$. The discriminant $d_K$ of an algebraic number field $K=\Q(\theta)$ and  the index of $\Z[\theta]$ in ring of integers $\Z_K$ are closely connected  by the well-known formula $D_f=[\Z_K:\mathbb Z[\theta]]^2d_K;$
where $D_f$ is the discriminant of the minimal polynomial $f(x)$ of $\theta$. In 1878,  Dedekind gave a criterion which provides  necessary and sufficient conditions for a prime $p$ to divide the index of $\Z[\theta]$ in $\Z_K$ (see Theorem \ref{dd}). 

 Let $K=\Q(\theta)$ be an algebraic number field where $\theta$ satisfies an irreducible quadrinomial $f(x)= x^{n} +ax^2 +bx+c\in\Z[x]$ with $b^2=4ac$. Then, in this paper using Dedekind's criterion, we characterize all the primes which divide the  group index  $[\Z_K:\Z[\theta]]$. This will provide  necessary and sufficient conditions involving only $a, c, n$ for $\{1, \theta, \cdots, \theta^{n-1}\}$ to be an integral basis of $K$. We also prove that if a prime $p\nmid b$ and divides the discriminant $d_K$ of $K$, then the exact power of $p$ dividing $d_K$ is one. As an application of this result, in Corollary \ref{1.1pc}, we provide an explicit formula for $d_K$ defined by a class of quadrinomials. In the end of the paper, we give examples which generate families of monogenic algebraic number fields. In some examples, we determine the index $[\Z_K:\Z[{\theta}]]$ as well.

Precisely, we prove:
\begin{theorem}\label{1.1}
Let $K=\mathbb Q(\theta)$ be an algebraic number field defiend by a root $\theta$ of an irreducible quadrinomial $f(x) = x^n+ax^2+bx+c \in \Z[x]$  with $b^2=4ac$. A prime $p$  dividing the discriminant $D_f$ of $f(x)$ does not divide $[\Z_K:\Z[\theta]]$ if and only if $p$ satisfies one of the following conditions:
\begin{itemize}
\item[(1)] If  $p\mid a$  and $p\mid c,$ then $p^2\nmid c$.
\item[(2)] If   $p\mid a$ and $p\nmid c$ with $r\geq 1$ as the highest power of $p$ dividing $n$, then either  $p\mid b_1$ and $p\nmid c_1$ or  $p\nmid b_1 [(-c_1)^n+c b_1^n]$ where  $c_1=\frac{{c +(-c)^{p^r}}}{p}$ and $b_1=\frac{b}{p}.$
\item[(3)] If  $p\nmid a$ and $p\mid c$ with $l\geq 1$ as the highest power of $p$ dividing $n-2$, then either $p\mid a_1$ and $p\nmid b_1$ or $p \nmid \bar{a}_1[(-\bar{a})^{2^{1-d}}{\bar{a}_1}^{\frac{n-2}{2^d}}-{(-\bar{b}_1)}^{\frac{n-2}{2^d}}]$ where  $a_1=\frac{{a+ (-a)^{p^l}}}{p}$, $b_1=\frac{b}{p}$ and $d$ equals $1$ or $0$ according as $2$ divides $n-2$ or not.
\item[(4)] If $p=2$  and $p\nmid ac$, then either $a\equiv 1 (\text{mod } 4)$ or $c\equiv 1 (\text{mod } 4)$.
\item[(5)] If $p\nmid b,$ then $p^2\nmid D_f.$
 \end{itemize}
\end{theorem}
The following corollaries are  immediate consequence of the above theorem. 
\begin{corollary}\label{cor}

Let $K=\Q(\theta)$, $f(x)$ be as in the above theorem. Then $\{1,\theta,\cdots,\theta^{n-1}\}$ is an integral basis of $K$ if and only  if each  prime $p$ dividing  $D_f$  satisfies one of the conditions $(1)-(5)$ of Theorem \ref{1.1}
\end{corollary}

The following result proved in \cite[Theorem 1.3]{jh-k}  can be quickly deduced from Corollary \ref{cor}.
\begin{corollary}\label{cor2}
Let $K = \Q(\theta)$ be an algebraic number field with $\theta$ a root of an irreducible polynomial $x^n-c\in\Z[x]$.
Then $\{1,\theta,\cdots,\theta^{n-1}\}$ is an integral basis of $K$ if and only  if
 $c$ is a square free integer and whenever $r\geq 1$ is the highest power of a prime $p$
dividing $n$, $p$ not dividing $c$, then $p^2$
 does not divide $c^{p^r}-c.$
\end{corollary}

\begin{corollary}\label{cor3} Let $f(x) = x^n+ax^2+bx+c \in \Z[x]$ be a monic polynomial of degree $n$ where $c\neq \pm 1$ is squarefree. Assume that the sets of primes dividing $a,b,c$ are the same. Let either $2|c$ or $a\equiv 1 (\text{mod } 4)$ or $c\equiv 1 (\text{mod } 4)$. Let $K = \Q(\theta)$ with $\theta$ a root of $f(x)$. Then $\Z_K =\Z[\theta]$ if and only if for each prime $p$ dividing $D_f$ and not dividing $a$, $p^2$ does not divide $D_f$.\end{corollary}

We now state the following result which shall be proved after the proof of Theorem \ref{1.1}. It is of independent interest as well.

\begin{proposition}\label{1.1p}
Let $K=\mathbb Q(\theta)$ be an algebraic number field where $\theta$ satisfies an irreducible quadrinomial $f(x) = x^n+ax^{2}+bx+c\in\Z[x]$ of degree $n \geq 3$  with $b^2 = 4ac$. If $p$ is a prime number such that $p\nmid b$ and $p$ divides the discriminant $d_K$ of $K$, then the exact power of $p$ dividing the discriminant $d_K$ of $K$ is one.
\end{proposition}
Using the fact
 that $D_f = [\Z_K : \Z[\theta]]^2d_K$, the above proposition and Theorem \ref{1.1} quickly yield the following explicit formula for the absolute value of $d_K$.
\begin{corollary}\label{1.1pc}
Let $K=\mathbb Q(\theta)$, $f(x)$ be as in Theorem \ref{1.1} and $|D_f| = \prod\limits_{i=1}^{k}p_i^{r_i}\prod\limits_{j=1}^{l}q_j^{t_j}$ be the prime factorization of $|D_f|$ into distinct prime numbers such that $p_i|b$ and $q_j\nmid b$. If all the primes $p_i$'s satisfy one of the conditions $(1)-(5)$ of Theorem \ref{1.1}, then the absolute value of the discriminant $d_K$ is given by 
$$|d_K| = \prod\limits_{i=1}^{k}p_i^{r_i}\prod\limits_{j=1}^{l}q_j^{\frac{1-(-1)^{t_j}}{2}}.$$
\end{corollary}
\section{Preliminary Results.}
\indent For two polynomials $f(x),g(x)$; $R(f,g)$ will denote the resultant of $f(x)$ and $g(x)$. Now we state some already known  properties of the resultant.
\begin{lemma}\label{lemma a} Let $f(x)=a_mx^m+a_{m-1}x^{m-1}+\cdots+a_0$ and $g(x)=b_nx^n+b_{n-1}x^{n-1}+\cdots+b_0$  be two polynomials with integer coefficients of degree $m$ and  $n$, respectively. Then\\
	{\bf (1)} $R(f,g)=(-1)^{mn}R(g,f).$\\
	{\bf (2)} If $f(x)=g(x)q(x)+r(x),$ then $$R(f,g)=(b_n)^{m-k}R(r,g)$$
	where  $k=\deg r(x).$\\
	{\bf (3)} $R(f,g)=(b_n)^m\prod\limits_{i=1}^n f(\gamma_i)$\\
	where $\gamma_1, \gamma_2,\cdots, \gamma_n$ are roots of $g(x)$.
	
\end{lemma}
\indent Note that the discriminant of a monic polynomial $f(x)\in\Z[x]$ and the resultant are related by the following formula
\begin{equation}\label{beta}
	D_f=(-1)^{\frac{n(n-1)}{2}}R(f',f)
\end{equation} where $f'$ is the derivative of $f.$

\begin{lemma}\label{lemma 1}
	Let $f(x)=x^n+ax^2+bx+c\in\Z[x]$ be  a quadrinomial of degree $n\ge3$ with $b^2=4ac.$ Then the following hold:
	\begin{enumerate}
		\item The discriminant of $f(x)$ is 
	\begin{equation}\label{df}
		D_f=(-1)^{\frac{n(n-1)}{2}}\left[n^n(-c)^{n-1}-4(n-2)^{n-2}\left(\frac{b}{2}\right)^n\right].
	\end{equation}
		\item  If $p\nmid b(n-2)$, then $D_f\equiv 0~(\text{mod}~ p^2)$ if and only if $f(-\frac{nb}{2a(n-2)})\equiv 0~(\text{mod}~ p^2).$
	\end{enumerate}
\end{lemma}
\begin{proof}
	{\bf (1)} We can take $a\neq 0$, because if $a = 0$, then using $b^2=4ac$ we have $b=0$. It is easy to check that the discriminant of $f(x) = x^n+c$ is $D_f = (-1)^{\frac{n(n-1)}{2}}n^n(-c)^{n-1}.$ So assume that $a\neq 0$. In view of Lemma \ref{lemma a}(1) and Equation \eqref{beta}, we have $D_f=(-1)^{\frac{n(n-1)}{2}}R(f,f').$ Using division algorithm, one can write $$f(x)=f'(x)q(x)+\frac{a(n-2)}{n}x^2+\frac{b(n-1)}{n}x+c.$$ We denote $\frac{a(n-2)}{n}x^2+\frac{b(n-1)}{n}x+c$ by $r(x)$.  Keeping in mind $(1)$ and  $(2)$ of Lemma \ref{lemma a}, we have $D_f=(-1)^{\frac{n(n-1)}{2}}n^{n-2}R(f',r).$ Since the roots of $r(x)$ are $\gamma_1=-\frac{b}{2a}$ and $\gamma_2=-\frac{nb}{2a(n-2)},$ therefore by $(3)$ of Lemma \ref{lemma a}, we have $$D_f=(-1)^{\frac{n(n-1)}{2}}n^{n-2} \left(\frac{a(n-2)}{n}\right)^{n-1}f'(\gamma_1)f'(\gamma_2).$$ A simple calculation shows that  $$D_f=(-1)^{\frac{n(n-1)}{2}}\left[n^n(-c)^{n-1}-4(n-2)^{n-2}\left(\frac{b}{2}\right)^n\right].$$
	{\bf (2)} Using hypothesis $b^2=4ac$, we have $2$ divides $b.$ Since $p\nmid b(n-2)$, it follows that $f\left(-\frac{nb}{2a(n-2)}\right)\equiv0~(\text{mod}~ p^2)$ if and only if $$\left(-\frac{nb}{2a(n-2)}\right)^n+a\left(-\frac{nb}{2a(n-2)}\right)^2+b\left(-\frac{nb}{2a(n-2)}\right)+c\equiv 0~(\text{mod}~ p^2).$$ Using $b^2=4ac$, the last equation can be rewritten as $$(-nb)^n+n^2c(2a)^n(n-2)^{n-2}-2nc(2a)^n(n-2)^{n-1}+c(2a(n-2))^n\equiv 0~(\text{mod}~ p^2).$$
	On dividing by $2^na^nc$ and simplifying, the last equation becomes $$ \frac{1}{c}\left(\frac{-nb}{2a}\right)^n+4(n-2)^{n-2}\equiv 0~(\text{mod}~ p^2).$$
	Since $2a=\frac{b^2}{2c},$ we see that $f\left(-\frac{nb}{2a(n-2)}\right)\equiv0\mod p^2$ if and only if $$n^n(-c)^{n-1}-4(n-2)^{n-2}\left(\frac{b}{2}\right)^n\equiv0 ~(\text{mod}~ p^2),$$ which completes the proof of the lemma in view of (1).
\end{proof}
The following  result will be used in the proof of Theorem \ref{1.1}.
\begin{lemma}\label{lemmanew}
Let $a,b,c, n$ be integers such that $2\nmid ac$,  $2|b$ and $2^2\nmid b$.  Let $n >2$ be an even integer and $F(x) = (\frac{a+a^2}{2})x^{2} + ( {ac+\frac{b}{2}})x + (\frac{c+c^2}{2})$ be a polynomial. Then the polynomial $F(x)$ is coprime to  $x^{\frac{n}{2}} + ax + c$ modulo $2$ if and only if $a \equiv 1$ $($mod $4)$ or $c \equiv 1$ $($mod $4)$.
\end{lemma}
\begin{proof}
Denote the polynomial $x^{\frac{n}{2}} + ax + c$ by $h(x)$. Since $2|n, 2|b, 2\nmid ac, 4\nmid b$, one can write $n  = 2m, a = 2A+1, b = 2B, c = 2C+1$ with $A, B, C, m \in \Z$ and $2\nmid B$. So
$$h(x)  \equiv x^{m} + x + 1~(\mbox{mod}~2),$$
and $F(x)$ can be written as
$$F(x) = (2A^2 + 3A + 1)x^2 + ((2A+1)(2C+1)+B)x + (2C^2 + 3C + 1),$$ which leads to 
$$F(x) \equiv (A+1)x^{2} +  C + 1~(\mbox{mod}~2).$$ Let $\alpha = A$ $($mod $2)$ and $\gamma = C$ $($mod $2)$. Then the values of $\alpha$ and $\gamma$ lie in $\{0, 1\}$. Hence it remains to prove the lemma in the following four cases.

\textbf{Case 1.} $\alpha = 0, \gamma = 0.$ In this case 
$$F(x) \equiv (x+1)^2 ~(\mbox{mod}~2),$$ and one can easily check that $F(x)$ is coprime to $h(x)$ modulo $2$.

\textbf{Case 2.} $\alpha = 0, \gamma = 1.$ Now we have
$$F(x) \equiv x^{2} ~(\mbox{mod}~2).$$ Therefore, we see that $F(x)$  is coprime to $h(x)$ modulo $2$. 

\textbf{Case 3.}  $\alpha = 1, \gamma = 1.$ In this case $${F}(x) \equiv 0~(\mbox{mod}~2).$$ So $h(x)$ divides $F(x)$ modulo $2$, hence $F(x)$ can not be coprime to $h(x)$ modulo $2$ in this situation.

\textbf{Case 4.} $\alpha = 1, \gamma = 0.$ Here we have $$F(x) \equiv 1  ~(\mbox{mod}~2),$$ which shows that $F(x)$ is coprime to $h(x)$ modulo $2$.

Therefore, in view of the above four cases, we see that $F(x)$ is not coprime to $h(x)$ modulo $2$ if and only if we are in Case 3, i.e., $\alpha = 1$ and $\gamma = 1.$ Returning to the definitions of $\alpha, \gamma$ we obtain that this happens if and only if one has $a \equiv 3$ (mod $4$) and $c \equiv 3$ (mod $4$), which completes the proof of the lemma.
\end{proof}
In 1878, Dedekind  gave the following  simple criterion known as Dedekind Criterion (cf.  \cite[Theorem 6.1.4]{HC}, \cite{RD}) which gave necessary and sufficient condition to be satisfied by $f(x)$ so that $p$ does not divide   $[\Z_K : Z[\theta]]$. 
\begin{theorem}\label{dd}
 (Dedekind Criterion) {\it Let $K=\mathbb{Q}(\theta)$ be an
algebraic number field with $f(x)$ as the minimal polynomial of
the algebraic integer $\theta$ over $\mathbb{Q}.$ Let $p$ be a
prime and
$\overline{f}(x) = \overline{g}_{1}(x)^{e_{1}}\ldots \overline{g}_{t}(x)^{e_{t}}$ be
the factorization of $\overline{f}(x)$ as a product of powers of
distinct irreducible polynomials over $\mathbb{Z}/p\mathbb{Z},$
with each $g_{i}(x)\in \mathbb{Z}[x]$ monic. Let $M(x)$ denote the polynomial
$\frac{1}{p}(f(x)-g_{1}(x)^{e_{1}}\ldots g_{t}(x)^{e_{t}})$ with
coefficients from $\mathbb{Z}.$ Then $p$ does not divide
$[\Z_{K}:\mathbb{Z}[\theta]]$ if and only if for each $i,$ we have
either $e_{i}=1$ or $\bar{g}_{i}(x)$ does not divide
$\overline{M}(x).$} 
\end{theorem}
We now state the following simple lemma proved in (\cite[Lemma 2.2]{JNT}), which will be used in the sequel. It can be easily proved using the Binomial Theorem.
\begin{lemma}\label{lala}

Let $h(x) = x^{m} + ax^{s} + b \in \Z[x]$ be a monic polynomial of degree $m$. Let $p$ be a prime and $k \in \N$. Then   
 $h(x^{p^{k}}) =  h(x)^{p^{k}} +   ph(x)T(x) + (ax^{sp^{k}} + b) + (-ax^{s} - b)^{p^{k}}$ for some polynomial $T(x) \in \Z[x]$.
 \end{lemma} 
 The following result will be used in the proof of Proposition \ref{1.1p} (cf. \cite[Theorem 4.24
 ]{Nar}).  Its proof is omitted.
\begin{theorem}\label{nana}
Let $p$ be a prime number and $K/\Q$ be an extension of degree $n$. Let $p\Z_K = \wp_1^{e_1}\cdots \wp_t^{e_t}$, where $\wp_1, \cdots, \wp_t$ are distinct prime ideals of $\Z_K$ and $N_{K/\Q}(\wp_i) = p^{f_i}$. If $(p, e_i)$ = 1 for $1 \leq i \leq t$, then the exact power of $p$ dividing $d_K$ is ${\sum\limits_{i=1}^{t}f_i(e_i-1)}$.
\end{theorem}
 \section{\textsc{Proof of Theorem \ref{1.1}}.}
 \begin{proof}[Proof of Theorem \ref{1.1}.]
   Consider the first case, when  $p\mid a $ and $p\mid c$. In this case  $f(x) \equiv x^n  ~(\text{mod} ~p)$. As in Theorem \ref{dd}, one can check that $M(x)=\frac{a}{p}x^2+\frac{b}{p}x+\frac{c}{p},$ hence we see that $p\nmid [\Z_K : \Z[\theta]]$ if and only if $x$ does not divide $\overline{M}(x)$. This is possible only when $p^2\nmid c.$ \\

  Consider now the case when $p\mid a$ and $p\nmid c$. Since $b^2=4ac$, we have $p|b$ and hence  $f(x)\equiv x^n+c ~(\text{mod}~ p).$ Keeping in mind the fact that $p|D_f$ and Equation $(\ref{df}),$ we see that $p\mid n.$  Write $n=p^rm,$ where $r\geq 1$ and  $p\nmid m.$  In view of Binomial theorem, we have 
  \begin{equation}\label{gt}
  	x^n+c\equiv (x^m+c)^{p^r}~(\text{mod}~ p).
  \end{equation}
   Let $\prod\limits_{i=1}^{t}\bar{g}_i(x)$ be the factorization of $x^m + \bar{c}$ over $\Z/p\Z$, where $g_1(x),g_2(x),\cdots,g_t(x)$ belonging to $\Z[x]$ are  monic polynomials which are distinct and irreducible modulo $p$. Write $x^m+c$ as $\prod\limits_{i=1}^{t}\bar{g}_i(x)+p A_1(x)$ for some polynomial $A_1(x)\in\Z[x]$. Then
 \begin{equation*}
  	f(x)=((x^{p^r})^m+c)+ax^2+bx=h(x^{p^r})+ax^2+bx
  \end{equation*}
where $h(x)=x^m+c.$ We now split this case into two sub cases according as $p$ is an odd prime or not.

First consider the sub case when $p\neq 2.$  Keeping in mind  that $p^2\mid a$ and applying Lemma \ref{lala} to $h(x),$ we see that
   \begin{equation}\label{2.5a}
  	f(x) = \big(\prod\limits_{i=1}^{t}g_i(x)+p A_1(x)\big)^{p^r} + pA(x) \prod\limits_{i=1}^{t}g_i(x) + p^2B(x) + c -c^{p^r} + bx
  \end{equation}
  for some polynomials $A(x), B(x) \in \Z[x]$. Denote $\frac{{c -c^{p^r}}}{p}$ and $\frac{b}{p}$ by $c_1$ and $b_1$, respectively.  In view of $(\ref{gt})$,  $\bar{f}(x)= \big(\prod\limits_{i=1}^{t}\bar{g}_i(x)\big)^{p^r}.$  Write $f(x)$ as $\prod\limits_{i=1}^{t}g_i(x)^{p^r} + pM(x)$, $M(x)\in \Z[x]$, $M(x)\in \Z[x]$. Keeping in mind that $r\geq 1$, it is immediate from (\ref{2.5a}) that   $$\overline{M}(x) =\bar{A}(x) \prod\limits_{i=1}^{t}\bar{g}_i(x) + \overline{c_1} + \overline{b_1}x.$$
  Applying Theorem \ref{dd}, $p\nmid [\Z_K:\Z_{\theta}]$ if and only $\overline{M}(x)$ is coprime $\prod\limits_{i=1}^{t}\bar{g}_i(x)$ which by virtue of the above equation holds if and only if  $(\overline{c_1} + \overline{b_1}x)$ is coprime to $\prod\limits_{i=1}^{t}\bar{g}_i(x)^{p^r} = \bar{h}(x^{p^r}) = x^n+\bar{c}$. This happens if and only if either  $p\mid b_1$ and $p\nmid c_1$ or  $p\nmid b_1 ((-c_1)^n+cb_1^n).$\\
  Now let $p=2.$ Since $b^2 = 4ac$, we see that $4 \mid a$ and $4 \mid b$. Therefore \begin{equation*}\label{2.5a1}
  	f(x) = \big(\prod\limits_{i=1}^{t}g_i(x)+2 A_1(x)\big)^{2^r} + 2C(x) \prod\limits_{i=1}^{t}g_i(x) + 4D(x) + c + c^{2^r} 
  \end{equation*}
  for some polynomials $C(x), D(x) \in \Z[x].$ Here  $1\le i\le t,$ $\bar{g}_i(x)\nmid \overline{M}(x)$ if and only if $\bar{g}_i(x)\nmid \overline{(\frac{{c +c^{2^r}}}{2})} .$ 
  This happens only when $4\nmid (c^{2^r}+c)$. 
  

 Consider now the case when $p\nmid a$ and $p\mid c.$ Note that $f(x)\equiv x^2(x^{n-2}+a)~(\text{mod}~ p).$  Let $l \geq 0$ denote the highest power of $p$ dividing $n-2$. Consider first the possibility when $l = 0$. If $p\ne 2,$ in view of  $b^2=4ac,$ we have $p^2\mid c.$ Since $x$ is the only repeated factor of $\bar{f}(x)$, we see that it divides $\overline{M}(x) =\frac{\bar{b}}{p}x$. Hence by Theorem \ref{dd}, $p$ divides $[\Z_K:\Z[\theta]]$. If $p=2,$ then observe that $p^2\mid b$ and $p^2\mid c$. Thus $\overline{M}(x)=\bar 0$. Therefore, in the case $l = 0$, $p$ always divides $[\Z_K:\Z[\theta]]$.
   Now assume that $l \geq 1$, say $n-2 = p^l s'$ and $p\nmid s'$. Write $x^{s'} + a = g_1(x)\cdots g_t(x) + pH(x)$, where $g_1(x), \cdots, g_{t}(x)$  are monic polynomials which are distinct as well as   irreducible modulo $p$ and $H(x) \in \Z[x]$. Keeping in mind that $p^2\mid c$ and applying Lemma \ref{lala} to $h(x)$ = $x^{s'} + a$, we can write $f(x)= x^2 (x^{n-2}  +a) +bx+c$ as
    \begin{equation}\label{eq:1303}
    f(x)=x^2\left[(\prod\limits_{i=1}^{t}g_i(x)+pH(x))^{p^l}+pT(x)\prod\limits_{i=1}^{t}g_i(x)+p^2U(x)+a+(-a)^{p^l}\right]+bx +c, 
   \end{equation}  
    \noindent where $T(x), U(x)$ belong to $\Z [x]$.                  
   Note that $x,\bar{g}_1(x) , \cdots, \bar{g}_t(x)$ are distinct irreducible factors of $\overline F(x)$. Denote $\frac{{a+(-a)^{p^l}}}{p}$ and $\frac{b}{p}$ by $a_1$ and $b_1$, respectively.  Since  $l\geq 1$,  by Theorem \ref{dd} $p\nmid [\Z_K:\Z[\theta]]$ if and only if $\bar{a}_1x^2 +\bar{b}_1$ and $x^{n-2}+\bar{a}$ are coprime. Note that if $p=2$, then $4|b$ and hence in this situation the polynomials $\bar{a}_1x^2 +\bar{b}_1$ and $x^{n-2}+\bar{a}$ will be coprime if and only if $2\nmid \bar{a}\bar{a}_1.$ Now assume that $p\neq 2$. In this situation, we see that $\bar{a}_1x^{2} + \bar{b}_1$ and $ x^{n-2} + \bar{a}$ are coprime if and only if either $($I$)$ $p \mid a_1$ and $p\nmid b_1$ or $($II$)$ $\bar{a}_1 \neq \bar{0}$ and $x^{2} + \frac{\bar{b}_1}{\bar{a}_1}$, $x^{n-2} + \bar{a}$ are coprime. Write $n-2=2^d s$ where $d$ equals $1$ or $0$ according as $2$ divides $n-2$ or not. Note that if $\xi$ is a common root of the polynomials  $x^{2} + \frac{\bar{b}_1}{\bar{a}_1}$, $ x^{n-2} + \bar{a}$   in the algebraic closure of $\Z/p\Z$, then we see that $(\frac{-\bar{b}_1}{\bar{a}_1})^{\frac{n-2}{2^d}}=(-\bar{a})^{2^{1-d}}$. Also a simple calculation shows that if $\xi^{2^{d}}  = (-\bar{a})^{-2^d}(-\frac{\bar{b}_1}{\bar{a}_1})^{\frac{n-1}{2^{1-d}}}$ with $({\frac{{-\bar{b}_1}}{\bar{a}_1}})^{\frac{n-2}{2^d}}=(-\bar{a})^{2^{1-d}}$, then $\xi$ is a  root of the polynomials $x^{2} + \frac{\bar{b}_1}{\bar{a}_1}$, $ x^{n-2} + \bar{a}$. Hence one can easily see that $($II$)$ holds if and only if $p \nmid \bar{a}_1[(-\bar{a})^{2^{1-d}}{\bar{a}_1}^{\frac{n-2}{2^d}}-{(-\bar{b}_1)}^{\frac{n-2}{2^d}}]$.


Consider now the case when $2\nmid ac.$ As $b^2=4ac,$  so $4\nmid b.$    Keeping in mind that $2|D_f$ and Equation $(\ref{df}),$ we see that $2\mid n.$  
  Write $n = 2m$. 
   Observe that $\bar{f}(x) = x^n + \bar{a}x^{2} + \bar{c}$. Denote $x^m + ax + c$ by $h(x)$ so that $f(x) = h(x^2) + bx.$ Let $h(x) \equiv \prod\limits_{i=1}^{t}g_i(x)^{r_i}$ $($mod $2)$ be the factorization of $h(x)$ into a product of irreducible polynomials modulo $2$ with $g_i(x)$ belonging to $\Z[x]$ monic and $r_i > 0$. Using Lemma \ref{lala}, we see that
  $$f(x) = h(x^2) + bx = (h(x))^2 + 2h(x)N(x) + ax^{2}+c + (-ax-c)^2 + bx$$ for some $N(x) \in \Z[x]$. Substituting $h(x) = \prod\limits_{i=1}^{t}g_i(x)^{r_i} + 2H(x)$ with $H(x) \in \Z[x]$ in the above equation, we see that there exists $N_1(x) \in \Z[x]$ such that 
  \begin{equation}
  f(x) = \big(\prod\limits_{i=1}^{t}g_i(x)^{r_i}\big)^{2} + 4N_1(x) + 2h(x)N(x)   + ax^{2}+c + (-ax-c)^2 + bx.
  \end{equation}
  Write $f(x) = \big(\prod\limits_{i=1}^{t}g_i(x)^{r_i}\big)^{2} + 2M(x)$ for some $M(x) \in \Z[x]$. Applying Theorem \ref{dd}, we see that $2 \nmid [\Z_K :\Z[\theta]]$ if and only if $\frac{1}{2}[ ax^{2}+c + (-ax-c)^2 + bx]$ is coprime to $h(x)$ modulo $2$, which in view of Lemma \ref{lemmanew} is equivalent to saying that either $a\equiv 1$ (mod $4$) or $c \equiv 1$ (mod $4$). \\ 

  Consider the last case when $p\nmid b.$ In view of $b^2=4ac$ and the fact that $p|D_f$, we see that $p\ne 2$ and $p\nmid n(n-2).$ One can easily check that $-(\overline{\frac{nb}{2(n-2)a}})$ is the only repeated root of $\bar{f}(x)$  and has multiplicity two.  Choose $d \in \Z$ such that
\begin{equation}\label{2.1a}
nbd \equiv -2(n-2)a ~~(\mbox{mod} ~p^2).
\end{equation}
Therefore, we have
\begin{equation}\label{tuntun}
\bar{f}(x) = (x-\bar{d})^2\bar{h}(x),
\end{equation}
  where $\bar{h}(x)$ belonging to $\Z/p\Z[x]$ is a separable polynomial. Write 
\begin{equation}\label{2.2a}
 f(x) = (x-d)q(x) + f(d)
 \end{equation}
  for some monic polynomial $q(x) \in \Z[x]$. Observe that $\bar{q}(x) = (x-\bar{d})\bar{h}(x)$ with $(x-\bar{d})$ does not divide $\bar{h}(x)$. Let $\bar{h}(x)  = \prod\limits_{i=1}^{t}\bar{g}_i(x)$ be the factorization of $\bar{h}(x)$ into a product of distinct irreducible polynomials over $\Z/p\Z$ with each $g_i(x) \in \Z[x]$ monic. Write 
$$ q(x) = (x-d) \prod\limits_{i=1}^{t}g_i(x) + ph_1(x)$$  
for some $h_1(x)$ belonging to $\Z[x]$. Substituting from the above equation in $(\ref{2.2a})$, we see that
 \begin{equation*}
  f(x) = (x-d)^2\prod\limits_{i=1}^{t}g_i(x) + p(x-d)h_1(x)+ f(d).
 \end{equation*}
Therefore applying Theorem \ref{dd}, we see that $p\nmid [\Z_K : \Z[\theta]]$ if and only if $p^2$ does not divide $f(d)$. Hence by Lemma \ref{lemma 1}, we have $p\nmid [\Z_K : \Z[\theta]]$ if and only if $p^2$ does not divide $D_f.$
So the proof of the theorem is complete. 
\end{proof}
\begin{proof}[Proof of Proposition \ref{1.1p}] Keeping in mind the formula $D_f = [\Z_K : \Z[\theta]]^2d_K$, we have  $p|D_f$ as $p$ divides $d_K$.  Now recall the proof of the last case of Theorem \ref{1.1}, we see that $p\neq 2$, $p\nmid n(n-2)$ and $$\bar{f}(x) = (x-\bar{d})^2\prod\limits_{i=1}^{t}\bar{g}_i(x),$$ where $x-\bar{d}, \bar{g}_1(x), \cdots, \bar{g}_t(x)$ are distinct and irreducible polynomials modulo $p$ with $g_i(x) \in \Z[x]$ monic for each $i$. Thus  using Hensel's Lemma, we see that $f(x)$ has a factorization in $\Q_p[x]$ such that 
$$f(x) = G(x)G_1(x)\cdots G_t(x),$$ where $G_1(x), \cdots, G_t(x)$ are distinct irreducible polynomials. 
The above factorization leads into the prime ideal factorization of $p\Z_K$ as $p\Z_K = \mathfrak{p} \wp_1 \cdots\wp_t$, where $\wp_i$'s are distinct prime ideals of $\Z_K$ with $N_{K/\Q}(\wp_i) = p^{\deg G_i(x)}$, and as $p$ dividing $d_K$ is ramified in $\Z_K$, we have $\mathfrak{p} = \wp^2$ for some prime ideal of $\Z_K$ with $N_{K/\Q}(\wp) = p$ and $\gcd(\wp, \wp_i) = \Z_K$.  Since $p\neq 2$, using Theorem \ref{nana}, we obtain that the exact power of $p$ dividing $d_K$ is one.
\end{proof}

\section{Examples.}
We now provide some examples which illustrate our results. In these examples, $K = \Q(\theta)$ with $\theta$ a root of $f(x)$.

\begin{example}
 Suppose that $c \neq \pm 1$ is a squarefree integer. Consider the polynomial $f(x) = x^7 + c(x^2+2x+1)$. Then the discriminant $D_f$ of $f(x)$ is given by $D_f=c^6[2^2\cdot 5^5c -7^7]$. By virtue of Theorem $\ref{1.1}(i)$, any prime $p$ dividing $c$  will not divide $[\Z_K : \Z[\theta]]$. Further in view of Corollary \ref{cor}, $\Z_K = \Z[\theta]$ if and only if  for each prime $p$ dividing $D_f$ and not dividing $c$, $p^2$ does not divide $(2^2\cdot 5^5c -7^7)$. Now  we calculate the exact value of $[\Z_K:\Z[\theta]]$ corresponding to some values of $c$.
 \begin{enumerate}
 	\item For $c = 2$, it can be easily seen that $|D_{f}| = 2^6\cdot 3^2\cdot 83\cdot 1069.$ In view of Theorem $\ref{1.1},$ $2,83,1069$ do not divide $[\Z_K:\Z[\theta]]$ and $3$ divides $[\Z_K:\Z[\theta]]$. Since $ D_{f} $=$[\Z_K : \Z[\theta]]^2\cdot d_K$, where $d_K$ is the discriminant of $K$, we see that $[\Z_K:\Z[\theta]]$ is $3$ in this case.
 	\item Consider $c = 5$. One can check that $|D_{f}| = 5^6\cdot 3\cdot 253681$. By Theorem $\ref{1.1},$ the primes $5, 3$ and $253681$ do not divide $[\Z_K:\Z[\theta]]$. So $[\Z_K:\Z[\theta]]=1$, and hence $\{1,\theta, \cdots, \theta^6\}$ is an integral basis of $K$ in this situation.
 	\item When $c = 7$, then one can see that the prime factorization of $|D_{f}|$ is given by $7^7\cdot 11^3\cdot 79$. Arguing as above, $7, 79$ do not divide $[\Z_K:\Z[\theta]]$  and $11$ divides $[\Z_K:\Z[\theta]]$. Therefore $[\Z_K:\Z[\theta]]=11$.
 \end{enumerate}
 \end{example}
\begin{example}
 Suppose that $c \equiv  1(\text {mod } 4)$  and consider an irreducible\footnote{Suppose there exists a prime factor $p$ of $c$ such that the highest power of $p$ dividing $c$ is $k$ and $\gcd(k,n) = 1$, then $f(x) = x^n+c(x^2+2x+1)\in\Z[x]$ will be irreducible over $\Q$ by Dumas irreduciblity criterion \cite{Dum}.} polynomial $f(x) = x^n + c(x^2+2x+1)$. Then we have $D_f=(-1)^{\frac{n(n-1)}{2}}c^{n-1}[(-1)^{n-1}n^n-4(n-2)^{n-2}c]$. By virtue of  Corollary \ref{cor}, $\Z_K = \Z[\theta]$ if and only if  for each prime $p$ dividing $D_f$, $p^2$ does not divide $c$ when $p\mid c$ and if $p\nmid c$, then $p^2\nmid D_f$. For example, if we take $n=5$ and $c\neq \pm 1$ is a squarefree integer, then $\{1,\theta,\theta^2,\theta^3,\theta^4\}$ is an integral basis of $K$ if and only if $(3125-108c)$ is squarefree. It can be verified that $(3125-108c)$ is squarefree when $c=-3, 5, 13, 17, 21$.
\end{example}

 \end{document}